\documentclass[reqno,11pt]{amsart}
\usepackage{amssymb,amsthm,amsfonts,amstext}
\usepackage{amsmath}
\usepackage{zito}
\usepackage{helvet}         
\usepackage{courier}        
\usepackage{fourier}
\usepackage{color}

\makeatletter \@addtoreset{equation}{section} \makeatother

\renewcommand\thefigure{\thesection.\@arabic\c@figure}
\renewcommand\thetable{\thesection.\@arabic\c@table}
\newtheorem{theorem}{Theorem}[section]
\newtheorem{lemma}[theorem]{Lemma}
\newtheorem{proposition}[theorem]{Proposition}
\newtheorem{corollary}[theorem]{Corollary}

\theoremstyle{remark}

\newcommand{\mc}[1]{{\mathcal #1}}

\newcommand{\bb}[1]{{\mathbb #1}}

\newcommand{\eps}{\varepsilon}

\title{Spectral gap inequality for long-range random walks}
\date{}
\author{Milton Jara}
\address{Instituto de Matem\'atica Pura e Aplicada, Estrada Dona Castorina 110, 22460-320  Rio de Janeiro, Brazil.}  \email{mjara@impa.br}

\begin{document}

\begin{abstract}
We show that the spectral gap of a random walk on the domain of normal attraction of an $\alpha$-stable law is of order $\mc O(n^{-\alpha})$ when restricted to boxes of size $n$. The proof is based on a comparison principle that may be of independent interest. The comparison principle also allows to derive a sharp bound on the spectral gap of exclusion and zero-range processes with long jumps when restricted to finite boxes in terms of the gap on the complete graph. 
\end{abstract}

\maketitle

A fundamental question in the theory of finite-state, reversible Markov chains is the estimation of the {\em spectral gap} of the chain, which is the difference between the first and second eigenvalue of the generator of the chain. The inverse of the gap, known as the {\em relaxation time} of the chain, measures the time that the chain need to equilibrate if it starts from a {\em typical} initial condition. The classical book \cite{LevPerWil} contains various examples of applications of this estimate and also various methods to obtain meaningful estimates of this spectral gap. 

One possible way to obtain estimates of the spectral gap of a chain, is by means of {\em comparison methods}, which allow to translate estimates of the gap in simpler chains to more complicated chains, see \cite{DiaS-C} for various examples of this technique.

It is well understood that for finite-range, symmetric random walks and conservative systems on $\bb Z^d/ n \bb Z^d$, like exclusion or zero-range processes, the spectral gap should be of order $\mc O(n^{-2})$, since by Brownian scaling, particles need times of order $\mc O(n^2)$ to discover the  geometry of the graph where they move. For random walks, this is an exercise, while for interacting particle systems it is a more delicate task. The corresponding estimates were proved in \cite{DiaS-C}, \cite{Qua} for the exclusion process and in \cite{LanSetVar}, \cite{BouCapD-PPos}, \cite{Mor} for the zero-range process. The proof goes by comparison with the case of systems evolving on the complete graph.

If instead of finite-range jumps we consider long-range jumps on the domain of normal attraction of an $\alpha$-stable law, the picture changes. The random walk now has a L\'evy scaling, and it needs times of order $\mc O(n^\alpha)$ to discover the geometry of the underlying graph. Therefore, the spectral gap should be of order $\mc O(n^{-\alpha})$. Our aim is to prove this fact for random walks evolving on finite intervals of the integer lattice. In order to fix ideas, let us consider a walk that jumps between $x$ and $y$ at instantaneous rate\footnote{We will take {\em continuous-time} Markov chains to avoid non-interesting periodicity issues} $|y-x|^{-(1+\alpha)}$. Surprisingly, comparing this walk with the walk evolving on the complete graph with constant rate $n^{-(1+\alpha)}$, one obtains the right bound for the spectral gap. The same holds true for interacting particle systems, and there is not much left to prove in that case. Therefore, in principle it should very easy to prove the same bound for general rates on the domain of attraction of an $\alpha$-stable law. But again surprisingly, the proof for general rates on te domain of attraction of $\alpha$-stable laws require a multiscale analysis inspired by rigorous renormalization group methods. 

This article is organized as follows. In Section \ref{s1} we define the random walks and the particle systems we will consider. In Section \ref{s2} we prove our main result, which is an abstract comparison theorem between expectations of subpolynomial functions with respect to measures on the domain of normal attraction of $\alpha$-stable laws. Although we could not think of any application different to the one stated here, we think that this comparison principle could be of interest in other situations. Once the comparison principle is proved, it is easy to derive the spectral gap inequality for the models stated in Section \ref{s1}. This is done in Section \ref{s3}. Finally, in Section \ref{s4} we discuss some applications of the spectral gap inequality. In particular, we focus on the derivation of stochastic fractional PDE's which was our original motivation.

\section{Definitions and results}
\label{s1}

\subsection{Continuous-time Markov chains}
\label{s1.1}

Let $\Omega$ be a finite set. We call $\Omega$ a {\em state space}. A {\em jump rate} is a function $r: \Omega \times \Omega \to [0,\infty)$ such that $r(x,x) =0$ for any $x \in \Omega$. Given  a jump rate $r(\cdot,\cdot)$ on a state space $\Omega$ we define $\gamma:\Omega \to [0,\infty)$ as
\[
\gamma (x) = \sum_{y \in \Omega} r(x,y).
\]
For each $x \in \Omega$ such that $\gamma(x) \neq 0$, we define $p(x,\cdot) : \Omega \to [0,1]$ as
\[
p(x,y) = \frac{r(x,y)}{\gamma(x)}
\]
for any $y \in \Omega$. If $\gamma(x) =0$ we define $p(x,x)=1$ and $p(x,y) =0$ for $y \neq x$. For any $x \in \Omega$, $p(x,\cdot)$ is a probability measure. The continuous-time Markov chain of state space $\Omega$ and rate $r(\cdot,\cdot)$ is the process $x(\cdot) = \{x(t); t \geq 0\}$ with the following dynamics. Whenever $x(t)$ is at a site $x$, it waits an exponential time of rate $\gamma(x)$, at the end of which it jumps to a site $y$ chosen with probability $p(x,\cdot)$, independently of the whole history of the process up to he moment of jump. After the jumps, a new independent, exponential time starts afresh, with rate $\gamma(y)$.

The generator $L$ of this Markov chain can be described explicitly in terms of $r(\cdot,\cdot)$. For each $f: \Omega \to \bb R$, $L f : \Omega \to \bb R$ is given by
\[
L f(x) := \sum_{y \in \Omega} r(x,y) \big( f(y) - f(x)\big)
\]
for any $x \in \Omega$. For each $x \in \Omega$, we denote by $\bb P^x$ the law of the chain $x(\cdot)$ with initial state $x_0=x$ and we denote by $\bb E^x$ the expectation with respect to $\bb P^x$. 

For each measure\footnote{In this article, all measures are {\em probability measures}.} $\nu$ in $\Omega$, we define
\[
\bb P^{\nu} = \sum_{x \in \Omega} \nu(x) \bb P^x.
\]
The measure $\bb P^\nu$ turns out to be the law of the chain $x(\cdot)$ with initial law $\nu$. We say that a measure $\mu$ is an {\em equilibrium measure} (or simply an {\em equilibrium}) of the chain $x(\cdot)$ if
\[
\bb P^\mu\big(x(t)=x\big) = \mu(x)
\]
for any $t \geq 0$ and any $x \in \Omega$. In other words, if $x(0)$ has law $\mu$, then $x(t)$ has law $\mu$ for any $t \geq 0$. This property is equivalent to the condition
\[
\int Lf d \mu = \text{ for any } f : \Omega \to \bb R.
\]
We say that the chain $x(\cdot)$ is {\em irreducible} if for any $x,y \in \Omega$ there exists a finite path $\{x_0=x,\dots,x_\ell=y\}$ such that $r(x_{i-1},x_i) >0$ for $i = 1,\dots,\ell$. An equivalent characterization of irreducibility is the following: for every $x,y \in \Omega$ and any $t > 0$, $\bb P^x\big(x(t)=y\big) >0$. From now on we assume that the chain $x(\cdot)$ is irreducible. In that case there exists a {\em unique} equilibrium measure $\mu$ of $x(\cdot)$. Notice that the second characterization of irreducibility implies that $\mu(x) >0$ for any $x \in \Omega$. 

A classical problem in the theory of Markov chains is the characterization of the convergence to equilibrium of the law of $x(t)$. Let us define the {\em transition probability kernel}  $\big\{P_t(x,y); t \geq 0; x, y \in \Omega\big\}$ of $\big\{x(t); t \geq 0\big\}$ as 
\[
P_t(x,y) = \bb P^x\big(x(t) = y\big).
\]
For any $x \in \Omega$ and any $t \geq 0$, the function $y \mapsto P_t(x,\cdot)$ defines a measure in $\Omega$ which corresponds to the law of $x(t)$ conditioned to $x_0=x$. A classical and not very difficult result states that for any continuous-time, irreducible Markov chain on a finite state space,
\[
P_t(x,y) \xrightarrow[t \to \infty]{} \mu(y)
\]
exponentially fast for any $x,y \in \Omega$. The exponential speed of convergence $\lambda$ can be computed as a Lyapounov exponent: for any $f: \Omega \to \bb R$ and any $t \geq 0$, let $P_tf : \Omega \to \bb R$ be defined as $P_t f(x) = \bb E^x[f(x(t))]$ for any $x \in \Omega$. Then,
\[
\lambda = -\sup_{f} \Big\{\limsup_{t \to \infty} \frac{1}{t} \log \| P_t f \|_{L^2(\mu)} \Big\}
\]
where the supremum runs over functions $f: \Omega \to \bb R$ such that $\int f d \mu =0$. 

It can be shown that the irreducibility of $\mu$ is equivalent to $0$ being a {\em simple} eigenvalue of $L$. Since $L$ is a finite-dimensional operator, this implies that $0$ is an isolated eigenvalue. The number $\lambda$ turns out to be the {\em spectral gap} of the generator $L$:
\[
\lambda = - \sup\big\{ \Re(v); v \text{ is an eigenvalue of } L \}.
\]

We say that the chain $x(\cdot)$ is {\em reversible} with respect to $\mu$ if ti satisfies {\em detailed balance}, that is,
\[
\mu(x) r (x,y) = \mu(y) r(y,x) \text{ for any } x,y \in \Omega.
\]
In that case, the operator $L$ is symmetric with respect to $\mu$ and the spectral gap can be computed by a variational formula. Let $\mc D(\cdot)$ be the Dirichlet form associated to the operator $L$:
\[
\mc D(f) = \int f (-Lf) d \mu = \tfrac{1}{2} \sum_{x,y \in \Omega} R(x,y) \big(f(x) -f(y)\big)
\]
for any $f: \Omega \to \bb R$. Then,
\begin{equation}
\label{varia}
\lambda = \inf \Big\{ \mc D(f); \|f\|_{L^2(\mu)} =1, \int f d \mu =0\Big\}.
\end{equation}
Our aim will be to establish sharp lower bounds on the spectral gap of some families of reversible Markov chains.

\subsection{Random walks and domains of normal attraction} We say that a non-zero function $p: \bb Z \to [0,\infty)$ is the {\em transition rate} of a random walk if $p(0) =0$ and $\sum_{z \in \bb Z} p(z) < \infty$. We say that a jump rate is symmetric if $p(z) = p(-z)$ for any $z \in \bb Z$. The {\em random walk} with transition rate $p(\cdot)$ is the continuous-time Markov chain $\{y_t; t \geq 0\}$ of jump rate given by $r(x,y) = p(y-x)$ for any $x,y \in \bb Z$. Since $\gamma(x) = \sum_{y \in \bb Z} p(y-x)$ is finite and independent of $x$, the construction of the previous section can be carried out despite of $\bb Z$ being infinite. 

We say that a symmetric transition rate $p(\cdot)$ is in the {\em domain of normal attraction} of an $\alpha$-stable law, $\alpha \in (0,2)$ if there exists a constant $c \in (0,\infty)$ such that 
\begin{equation}
\label{antofagasta}
\lim_{x \to \infty} x^\alpha \sum_{y \geq x} p(y) = c.
\end{equation}
We denote this by $p(\cdot) \in \mathrm{DAN}(\alpha)$\footnote{Since DNA means something else, we use the spanish acronym for {\em dominio de atracci\'on normal}}. The pertinence of this definition can be seen in the following proposition, which is a part of the Gnedenko-Kolmogorov theorem:

\begin{proposition}
\label{p1}
Let $p(\cdot)$ be a symmetric transition rate. Let $\alpha \in (0,2)$and let $\{y(t); t \geq 0\}$ be the random walk with rate $p(\cdot)$. The sequence $\{t^{-1/\alpha}y(t); t \geq 0\}$ has a non-trivial limit in law if and only if $p(\cdot) \in \mathrm{DAN}(\alpha)$. In that case, the limiting law is a symmetric, $\alpha$-stable law.
\end{proposition}

Since an $\alpha$-stable law is absolutely continuous, $\bb P^x\big(y(t)=y\big) \to 0$ as $t \to \infty$ for any $x,y \in \bb Z$. Therefore, the question about the convergence to equilibrium of this chain is not well posed. Nevertheless, this question makes sense if we restrict the walk to a {\em finite} interval. 

For each $n \in \bb N$, let $\Lambda_n = \{-n,\dots,n\}$ be the box of radius $n$ and centered at the origin in $\bb Z$. Let $p(\cdot)$ be a symmetric transition rate and for simplicity assume that $p(1)>0$. Let $\{y_n(t); t \geq 0\}$ be the Markov chain with state space $\Lambda_n$ and jump rate $r(x,y) = p(y-x)$. In other words, $\{y_n(t), t \geq 0\}$ is the walk restricted to the set $\Lambda_n$ mentioned above. The condition $p(1) >0$ implies that the chain $\{y_n(t); t \geq 0\}$ is irreducible, and therefore it has a unique equilibrium. Since $p(\cdot)$ is symmetric, the uniform measure $\mu_n$ in $\Lambda_n$ satisfies detailed balance, and therefore it is the unique equilibrium of this chain. The main result of this article is the following

\begin{theorem}
\label{t1}
Let $p(\cdot)$ be a symmetric transition rate. Assume that $p(\cdot) \in \mathrm{DAN}(\alpha)$ for some $\alpha \in (0,2)$ and also assume that $p(1) >0$. Let $\lambda_n$ be the spectral gap of the chain $\{y_t^n; t \geq 0\}$ described above. There exist constants $\kappa_1, \kappa_2 \in (0,\infty)$ such that
\begin{equation}
\label{SG}
\kappa_1 (2n+1)^{-\alpha} \leq \lambda_n \leq \kappa_2 (2n+1)^{-\alpha} 
\end{equation}
for any $n \in \bb N$.
\end{theorem}

The estimate \eqref{SG} is known in the literature as the {\em spectral gap inequality} for the family of chains $\{y_n(t); t \geq 0\}_{n \in \bb N}$.
We present here a first application of this theorem, which relates the polynomial decay of the transition probability kernel of the random walk in $\bb Z$ to the spectral gap on finite boxes:

\begin{corollary}
\label{c1}
Let $p(\cdot) \in \mathrm{DAN}(\alpha)$ and let $\{y(t); t \geq 0\}$ be the random walk with transition rates $p(\cdot)$. Assume that $p(1) >0$. Then there exists a constant $c_0 \in (0,\infty)$ such that
\[
\bb P^0(y(t) = 0 ) \leq \frac{c_0}{t^{1/\alpha}}.
\] 
\end{corollary}

In the next sections we will present other applications, which were actually our initial motivation to consider this problem.

\subsection{The exclusion process}

Let us consider the following dynamics. Let $p(\cdot)$ be a transition rate. Initially, particles are placed in $\Lambda_n$ in such a way that there is at most one particle per site. The particles follow independent Markov chains with rates $r(x,y)=p(y-x)$, except for the so-called {\em exclusion rule}: each time a particle tries to jump to a site already occupied by another particle, the jump is suppressed. The stochastic process $\{\eta^n(t); t \geq 0\}$ generated by this dynamics is known in the literature as the {\em exclusion process} with transition rate $p(\cdot)$. In order to describe in a more rigorous way the process generated by this formal description we need some notation. Define $\Omega_n^{\mathrm{ex}} = \{0,1\}^{\Lambda_n}$. For $\eta \in \Omega_n^{\mathrm{ex}}$ and $x,y \in \Lambda_n$, let $\eta^{x,y} \in \Omega_n^{\mathrm{ex}}$ be given by
\[
\eta^{x,y}_z=
\left\{
\begin{array}{r@{\;;\;}l}
\eta_y & z=x\\
\eta_x & z=y\\
\eta_z & z \neq x,y.
\end{array}
\right.
\]
For $f: \Omega_n^{\mathrm{ex}} \to \bb R$ we define $L_n^{\mathrm{ex}} f: \Omega_n^{\mathrm{ex}} \to \bb R$ as
\[
L_n^{\mathrm{ex}} f(\eta) := \sum_{x,y \in \Lambda_n} p(y-x) \eta_x (1-\eta_y) \big( f(\eta^{x,y})-f(\eta)\big)
\]
for any $\eta \in \Omega_n^{\mathrm{ex}}$. The process $\eta^n(\cdot) = \{\eta^n(t); t \geq 0\}$ turns out to be the continuous-time Markov chain generated by the operator $L_n^{\mathrm{ex}}$. This chain is not irreducible: the initial number of particles is left unchanged by the dynamics. However, if we assume that $p(1)p(-1) >0$, then the chain is irreducible on each of the sets
\[
\Omega_{n,\ell}^{\mathrm{ex}} := \Big\{ \eta \in \Omega_n^{\mathrm{ex}}; \sum_{x \in \Lambda_n} \eta_x = \ell\Big\},
\]
where $\ell =0,1,\dots,2n+1$. If in addition we assume that $p(\cdot)$ is symmetric, then on each of the sets $\Omega_{n,\ell}^{\mathrm{ex}}$ the unique equilibrium measure of the exclusion process in $\Omega_{n,\ell}^{\mathrm{ex}}$ is the uniform measure.

Let $\lambda_{n,k}^{\mathrm{ex}}$ be the spectral gap of the process $\eta(\cdot)$ restricted to the set $\Omega_{n,\ell}^{\mathrm{ex}}$. We have the following result:

\begin{theorem}
\label{t2} Let $p(\cdot)$ be a symmetric transition rate. Assume that $p(1) >0$ and that $p(\cdot) \in \mathrm{DAN}(\alpha)$ for some $\alpha \in (0,2)$. Let $\lambda_{n,\ell}^{\mathrm{ex}}$ be the spectral gap of the exclusion process with transition rate $p(\cdot)$ in $\Omega_{n,\ell}^{\mathrm{ex}}$. Then there exist constants $\kappa_1^{\mathrm{ex}}, \kappa_2^{\mathrm{ex}} \in (0,\infty)$ such that
\[
\kappa_1^{\mathrm{ex}} (2n+1)^{-\alpha} \leq \lambda_{n,\ell}^{\mathrm{ex}} \leq \kappa_2^{\mathrm{ex}} (2n+1)^{-\alpha}
\]
for any $n \in \bb N$ and any $\ell \in \{0,1,\dots,2n+1\}$.
\end{theorem}

The easiest way to prove this theorem is using {\em Aldous' conjecture}, which states that in any weighted graph, the spectral gap of the exclusion process is equal to the spectral gap of the underlying random walk. This conjecture was proved by Caputo, Liggett and Richthammer in \cite{CapLigRic}. In our context, Theorem \ref{t2} is an immediate consequence of Theorem \ref{t1} and Aldous' conjecture.

\subsection{The zero-range process} Let $g: \bb N_0 \to [0,\infty)$ be such that $g(0) =0$ and $g(k) > 0$ whenever $k \neq 0$. We assume that $g(\cdot)$ satisfies {\em Andjel's condition}:
\[
\sup_{k \in \bb N_0} \big| g(k+1) -g(k)\big| < +\infty.
\]
Andjel's condition is needed in order to guarantee the existence of the zero-range dynamics in infinite volume. It is also needed to prove various results even in finite volume, so it is a standard assumption in the literature. 
Let $\Omega_n^{\mathrm{zr}} = \bb N_0^{\Lambda_n}$ be the state space of a Markov process which we describe below. For $x,y \in \Lambda_n$ and $\xi \in \Omega_n^{\mathrm{zr}}$ such that $\xi_x \leq 1$, define $\xi^{x,y} \in \Omega_n^{\mathrm{zr}}$ as
\[
\xi^{x,y}_z = 
\left\{
\begin{array}{c@{\;;\;}l}
\xi_x-1 & z =x\\
\xi_y+1 & z =y\\
\xi_z & z \neq x,y.
\end{array}
\right.
\]
Let $p(\cdot)$ as in Theorem \ref{t1}. For each $f: \Omega_n^{\mathrm{zr}} \to \bb R$ we define $L_n^{\mathrm{zr}} f: \Omega_n^{\mathrm{zr}} \to \bb R$ as
\[
L_n^{\mathrm{zr}} f(\xi) := \sum_{x,y \in \Lambda_n} p(y-x) g(\xi_x) \big( f(\xi^{x,y})-f(\xi)\big).\footnote{Here we use the convention $g(\xi_x) f(\xi^{x,y})=0$ whenever $\xi_x=0$.}
\]
The {\em zero-range process } with interaction rate $g(\cdot)$ and transition rate $p(\cdot)$ is the continuous-time Markov chain $\{\xi^n(t); t \geq 0\}$ in $\Omega_n^{\mathrm{zr}}$ generated by the operator $L_n^{\mathrm{zr}}$. The dynamics of this process is the following. A particle jumps from $x$ to $y$ at instantaneous rate $g(\xi_x) p(y-x)$. This happens independently for each couple $x,y \in \Lambda_n$. Since $p(1) >0$, this chain is ergodic on each of the sets
\[
\Omega_{n,\ell}^{\mathrm{zr}} := \Big\{ \xi \in \Omega_n^{\mathrm{zr}}; \sum_{x \in \Lambda_n} \xi_x = \ell\Big\},
\]
and the zero-range process is in effect a chain in a finite-state space. Its unique equilibrium measure can be described as follows. Define $g(0)!=1$ and $g(\ell)! = g(1) \dots g(\ell)$ for $\ell \geq 1$. Then, the measure $\mu_{n,\ell}^{\mathrm{zr}}$ in $\Omega_{n,\ell}^{\mathrm{zr}}$ defined as
\[
\mu_{n,\ell}^{\mathrm{zr}}(\xi) : = \frac{1}{Z_{n,\ell}} \prod_{x \in \Lambda_n} \frac{1}{g(\xi_x)!} 
\]
for any $\xi \in \Omega_{n,\ell}^{\mathrm{zr}}$ is the equilibrium measure of the zero-range process in $\Omega_{n,\ell}^{\mathrm{zr}}$.  Here $Z_{n,\ell}$ is the normalization constant, which is finite since $\Omega_{n,\ell}^{\mathrm{zr}}$ is finite. We have the following result:

\begin{theorem}
\label{t3}
Let $p(\cdot)$ be a symmetric transition rate such that $p(1) >0$. Assume that $p(\cdot) \in \mathrm{DAN}(\alpha)$ for some $\alpha \in (0,2)$. Let $\lambda_{n,\ell}^{\mathrm{zr}}$ be the spectral gap of the zero-range process of interaction rate $g(\cdot)$ and transition rate $p(\cdot)$ in $\Omega_{n,\ell}$. There exists a constant $\kappa_0^{\mathrm{zr}} \in (0,\infty)$ such that
\begin{itemize}
\item[i)] whenever there exists $\eps_0 >0$ and $\ell_0 \in \bb N$ such that $g(\ell+\ell_0) > g(\ell) + \eps_0$ for any $\ell \in \bb N_0$,
\[
\kappa_0^{\mathrm{zr}} (2n+1)^{-\alpha} \leq \lambda_{n,\ell}^{\mathrm{zr}}
\]
for any $n, \ell \in \bb N$,
\item[ii)] if $g(k) = \mathbb{1}\{k \geq 1\}$, then
\[
\kappa_0^{\mathrm{zr}} (2n+1)^{-\alpha}(1+\rho)^{-2} \leq \lambda_{n,\ell}^{\mathrm{zr}}
\]
for any $n, \ell \in \bb N$, where $\rho = \frac{\ell}{2n+1}$. 
\end{itemize}
\end{theorem}

\section{The comparison principle}
\label{s2}

The proofs of Theorems \ref{t1}, \ref{t2} and \ref{t3} are based on a general {\em comparison principle}, which allows to reduce the proof to the case on which we replace $p(\cdot)$ by the more canonical transition rate $q(z) = |z|^{-(1+\alpha)}$. In this section we prove such comparison principle, which turns out to be very general and it can be of independent interest. In order to state this comparison principle, we need t introduce some definitions. We say that a function $\phi: \bb N \to [0,\infty)$ is {\em $K$-subpolynomial} if there exists $K$ finite such that
\begin{equation}
\label{caldera}
\phi(x+y) \leq K\big(\phi(x) + \phi(y)\big)
\end{equation}
for any $x,y \in \bb N$. Notice that any $K$-subpolynomial function $\phi$ satisfies $\phi(x) \leq C x^\nu$ for $\nu = 1+\frac{\log K}{\log 2}$ and some finite constant $C$, justifying the denomination subpolynomial. Notice as well that any non-trivial, $K$-subpolynomial function satisfies $\phi(1) >0$.

\begin{theorem}
\label{t4}
Let $p(\cdot)$ be a symmetric transition rate such that $p(\cdot) \in \mathrm{DAN}(\alpha)$ for some $\alpha \in (0,2)$. Assume that $p(1)>0$. For any $K>0$ there exists a constant $\kappa = \kappa(K,p(\cdot))$ such that for any $K$-subpolynomial function $\phi: \bb N \to \bb R$ and any $n \in \bb N$,
\begin{equation}
\label{copiapo}
\sum_{z=1}^n q(z) \phi(z) \leq \kappa \sum_{z=1}^n p(z) \phi(z), 
\end{equation}
where $q(z):= z^{-(1+\alpha)}$.
\end{theorem}

Before we start the proof of Theorem \ref{t4}, let us explain why such a result can be difficult to prove. Let $\{z_\ell; \ell \in \bb N\}$ be an increasing sequence of natural numbers such that $z_\ell \to \infty$ and $\frac{z_{\ell+1}}{z_\ell} \to1$ as $\ell \to \infty$. Then define
\[
p(\pm z_\ell) = \frac{1}{z_\ell^\alpha} - \frac{1}{z_{\ell+1}^\alpha}
\]
for $\ell \in \bb N$ and $p(\pm z) =0$ otherwise. The transition rate $p(\cdot) \in \mathrm{DAN}(\alpha)$, but it has huge gaps on which $p(z) =0$. In consequence, on the right-hand side of \eqref{copiapo}, $\phi(z)$ does not appear for a whole bunch of points $z$. Filling these gaps would require the use of \eqref{caldera} in a very careful way, to not overuse small values of $z$. We were able to construct by hand a proof of Theorem \ref{t4} for this particular choice of $p(\cdot)$, but it turned out to be too much dependent on the particular structure of $p(\cdot)$. Therefore, we will proceed with a different idea.

\begin{proof}[Proof of Theorem \ref{t4}]
The idea is to use a renormalization group approach. In a sense, the most regular transition rate in $\mathrm{DAN}(\alpha)$ should be $q_0(\cdot)$ defined as $q_0(z) = z^{-\alpha} -(z+1)^{-\alpha}$ for any $z >0$. 
Notice that $q_0(z) = \alpha |z|^{-(1+\alpha)} + \mc O\big(|z|^{-(2+\alpha)}\big)$ and therefore $q_0(\cdot)$ is equivalent to $q(\cdot)$. This means that \eqref{copiapo} holds for $q(\cdot)$ if and only if it holds for $q_0(\cdot)$, with maybe a different constant $\kappa$.

It will be convenient to introduce some notation. For $A \subseteq \bb N$ we define $p(A) = \sum_{x \in A} p(x)$ and
\[
\mc D_p(A) = \sum_{x \in A} p(x) \phi(x).
\]
We define $q(A)$ and $\mc D_q(A)$ replacing $p(\cdot)$ by $q(\cdot)$ in the definitions above. Theorem \ref{t4} will be proved if we show that there exist constants $\kappa_1 \in (0,\infty)$ and $\theta \in (0,1)$ such that
\[
\mc D_q\big( \{1,\dots,n\}\big) \leq \kappa_1 \mc D_p\big(\{1,\dots,n\}\big) + \theta \mc D_q\big(\{1,\dots,n\}\big) 
\]
for any $\phi$ $K$-subpolynomial and any $n \geq n_0$. The number $n_0$ may depend on $K$ and $p(\cdot)$, but not on $\phi$ or $n$. 

Let $\{b_n; n \in \bb N\}$ be a strictly increasing sequence in $\bb N$. Assume that there is a constant $b \in (1,\frac{1+\sqrt 5}{2})$ such that
\[
\lim_{n \to \infty} \frac{b_n}{b^n} =1.
\]
Define $a_n = b_n+b_{n+1}-b_{n+2}$. Notice that
\[
\lim_{n \to \infty} \frac{a_n}{b_n} = 1+b-b^2 =:a
\]
belongs to the interval $(0,1)$. Notice as well that
\[
\lim_{n \to \infty} \frac{a_{n+1}-a_n}{a_n} = b-1>0.
\]
Therefore, $a_n$ is positive and strictly increasing for $n \geq n_0$. For simplicity, we assume that $a_n$ is positive and strictly increasing for any $n$, a condition that is satisfied after considering $b_{n-n_0}$ instead of $b_n$. Define the intervals
\[
A_n = \{a_n+1,\dots,b_n\}, \quad B_n =\{b_{n+1}+1,\dots,b_{n+2}\},
\]
\[
D_n = \{b_{n+1}-b_n+1,\dots,b_{n+2}-a_n-1\}.
\]
Recall \eqref{antofagasta}. Absorbing the constant $c$ into $\kappa$, wlog\footnote{{\em wlog} stands for {\em without loss of generality}.} we can assume that $c=1$. Notice that
\begin{equation}
\label{coquimbo}
\lim_{n \to \infty} b^{n} \big( p(A_n),p(B_n),p(D_n)\big) = \Big( \frac{1}{a^\alpha}-1, \frac{1}{b^\alpha}-\frac{1}{b^{2\alpha}}, \frac{1}{(b-1)^\alpha}-\frac{1}{(b^2-a)^\alpha}\Big).
\end{equation}
Notice as well that
\[
\lim_{x \to \infty} x^\alpha \sum_{y \geq x} q(y) = \frac{1}{\alpha},
\]
and therefore the limits in \eqref{coquimbo} are multiplied by $\alpha$ when we consider $q(\cdot)$ instead of $p(\cdot)$. By construction, $x-y \in D_n$ whenever $x \in B_n$ and $y \in A_n$. We have that
\[
\phi(x) \leq K\big( \phi(y) + \phi(x-y)\big).
\]
Multiplying this estimate by $p(y)$ and summing over $y \in A_n$ we get the bound
\[
p(A_n) \phi(x) \leq K \mc D_p(A_n) + K \sum_{y \in A_n} p(y) \phi(x-y).
\]
Multiplying this estimate by $q(x)$ and summing over $x \in B_n$, we get the bound
\begin{equation}
\label{vicuna}
p(A_n) \mc D_q(B_n) \leq K q(B_n) \mc D_p(A_n) + K \sum_{\substack{x \in B_n\\y \in A_n}} q(x) p(y) \phi(x-y).
\end{equation}
The sum on the right-hand side of this estimate can be written as
\[
\sum_{z \in D_n} \theta_n(z) q(z) \phi(z), \text{ where } \theta_n(z) := \sum_{\substack{y \in A_n:\\y+z \in B_n}} \frac{q(y+z)p(y)}{q(z)}.
\]
Using the monotonicity of $q(\cdot)$ we see that
\[
\theta_n(z) \leq \sum_{y \in A_n} \frac{q(b_{n+1})p(y)}{q(b_{n+2}-a_n)} = \frac{q(b_{n+1}) p(A_n)}{q(b_{n+2}-a_n)},
\]
from where
\[
\sum_{z \in D_n} \theta_n(z) q(z) \phi(z) \leq \frac{q(b_{n+1}) p(A_n)}{q(b_{n+2}-a_n)} \mc D_q(D_n).
\]
Putting this estimate back into \eqref{vicuna}, we conclude that
\[
\mc D_q(B_n) \leq \frac{K q(B_n)}{p(A_n)} \mc D_p(A_n) +  \frac{K q(b_{n+1})}{q(b_{n+2}-a_n)} \mc D_q(D_n)
\]
as soon as $p(A_n) >0$. Notice that $b^{\alpha n} p(A_n) \to a^{-\alpha}-1>0$ as $n \to \infty$. Therefore, there exists $n_1$ finite such that $p(A_n) >0$ for any $n \geq n_1$. Considering the sequence $\{b_{n-n_1}; n \in \bb N\}$ instead of $\{b_n; n \in \bb N\}$, we can assume wlog that $p(A_n)>0$ for any $n$. Notice that $b^2-a = (2b+1)(b-1)$. Therefore,
\[
\Gamma_1 := \lim_{n \to \infty} \frac{q(B_n)}{p(A_n)} = \frac{b^{-\alpha} - b^{-2\alpha}}{\alpha(a^{-\alpha}-1)}
\]
and
\[
\Gamma_2:= \lim_{n \to \infty} \frac{q(b_{n+1})}{q(b_{n+2}-a_n)} = \Big( \frac{(b-1)(2b+1)}{b}\Big)^{1+\alpha}.
\]
We conclude that there exists $n_2 \geq n_1$ such that
\begin{equation}
\label{tongoy}
\mc D_q(B_n) \leq 2 K\big( \Gamma_1\mc D_p(A_n) +  \Gamma_2 \mc D_q(D_n) \big)
\end{equation}
for any $n \geq n_2$. What it is important is that the constants $\Gamma_1, \Gamma_2$ do not depend on $n$. When $b \to 1$, $\Gamma_1 \to \alpha^{-1}$, while $\Gamma_2 \to 0$. In fact,
\begin{equation}
\label{sanfelipe}
\lim_{b \to 1} \frac{\Gamma_2}{(b-1)^{1+\alpha}} = 3^{1+\alpha}.
\end{equation}
This fact will play a crucial role in a few lines. Define $I_n = \{1,\dots,b_n\}$. Notice that the sets $B_n$ are disjoints, but the sets $A_n$, $D_n$ are not. Let $m \geq n_2$ be a constant to be chosen later on.  Notice that $\cup_{j=m}^n B_j = I_{n+2} \setminus I_{m+1}$. Summing estimates \eqref{tongoy} over $\{m,\dots,n\}$ we obtain the bound
\begin{equation}
\label{losandes}
\mc D_q(I_{n+2}\setminus I_{m+1}) \leq 2K\big( \ell_1 \Gamma_1 \mc D_p(I_n) + \ell_2 \Gamma_2 \mc D_q(I_{n+2})\big),
\end{equation}
where the constants $\ell_1$, $\ell_2$ count the number of overlaps of the sets $\{A_j; j=m,\dots,n\}$, $\{D_j; j=m,\dots,n\}$. Let us proceed to compute the constants $\ell_1$, $\ell_2$. Since $\frac{a_n}{b_n} \to a>0$ as $n \to \infty$, for any $\ell \in \bb N$ such that $b^{-\ell} <a$, $z$ belongs to at most $\ell$ sets $A_j$ for any $z$ large enough. Define
\[
\ell_1 = \Big \lceil \frac{-\log a}{\log b} \Big \rceil +1
\]
and take $n_3 \geq n_2$ such that the intervals $A_n$, $A_{n+\ell_1}$ do not overlap for any $n \geq n_3$.
In the same way, if 
\[
b_{n+\ell+1} -b_{n+\ell} > b_{n+2} -a_n,
\]
then the intervals $D_n$, $D_{n+\ell}$ do not overlap. When $n \to \infty$, this condition becomes $b^\ell > 2b+1$, so if we choose
\[
\ell_2 = \Big\lceil \frac{\log(2b+1)}{\log b} \Big\rceil,
\]
we see that there exists $n_4 \geq n_3$ such that $z$ belongs to at most $\ell$ sets $D_j$ for any $z \geq b_{n_3+2} - a_{n_3}$. Therefore, we can choose $m = n_4$ and \eqref{losandes} holds true for any $n \geq m$. Notice that 
\[
\lim_{b \to 1} (b-1) \Big\lceil \frac{\log(2b+1)}{\log b} \Big\rceil = \log 3
\]
and
\[
\lim_{b \to 1}  \Big \lceil \frac{-\log a}{\log b} \Big \rceil =1.
\]
Observe that $\ell_2 \Gamma_2 \to 0$ as $b \to 1$, thanks to the exponent $1+\alpha$ in \eqref{sanfelipe}. Now we need to choose all the constants in the right order. Let $\delta \in (0, \frac{-1+\sqrt 5}{2})$ be such that $\theta := 2K \ell_2 \Gamma_2 <1$ for $b=1+\delta$. Let $k_0$ be such that $b^k > \frac{2}{\delta}$. Define $b_n = \lfloor b^{n+k_0} \rfloor$. The condition on $k_0$ ensures that $\{b_n; n \in \bb N\}$ is strictly increasing. Then, there exists $m = m(b, p(\cdot))$ 
\[
\mc D_q( I_{n+2} \setminus I_{m+1}) \leq 2 K  \ell_1 \Gamma_1 \mc D_p(I_{n}) + \theta \mc D_q(I_{n+2}).
\]
No w we just need to estimate $\mc D_q(I_{m+1})$. This is a finite problem, since $m$ is fixed. This estimate becomes very simple if we admit the following lemma:

\begin{lemma}
\label{l1}
Let $K>0$ be fixed. Then for any $K$-subpolynomial function $\phi: \bb N \to [0,\infty)$, 
\[
\phi(x) \leq 2K \phi(1) x^\nu
\]
for any $x \in \bb N$, where $\nu = 1+ \frac{\log K}{\log 2}$.
\end{lemma}
Using the lemma,
\[
\mc D_q(I_{m+1}) = \sum_{x=1}^{b_{m+1}} \frac{\phi(x)}{x^{1+\alpha}} \leq \sum_{x=1}^{b_{m+1}} C \phi(1) x^{\nu-1-\alpha} \leq C_1 \phi(1),
\]
where $C_1$ depends on $K$, $\alpha$ and $m$. But $\phi(1) \leq p(1)^{-1} \mc D_p(I_n)$ for any $n \in \bb N$. We conclude that
\[
\mc D_q(I_{n+2})\leq \frac{2 K \ell_1 \Gamma_1+C_1}{1-\theta} \mc D_p(\mc I_n),
\]
which proves the theorem.
\end{proof}

\begin{proof}[Proof of Lemma \ref{l1}]
Let $\Theta_k := \sup_{x \leq 2^k} \phi(x)$. Then, any $x \leq 2^{k+1}$ can be written as $x = y+z$ with $y, z \leq 2^k$. Therefore, $\phi(x) \leq K(\phi(y)+\phi(z))$, from where we conclude that $\Theta_{k+1} \leq 2 K \Theta_k$. Inductively, $\Theta_k \leq (2K)^k\Theta_0= (2K)^k \phi(1)$. Therefore, if $2^k < x \leq 2^{k+1}$,  
\[
\phi(x) \leq \Theta_{k+1} \leq (2K)^{k+1} \phi(1) \leq 2K \phi(1) 2^{k\frac{\log 2K}{\log 2}} \leq 2K \phi(1) x^\nu,
\]
as we wanted to show.
\end{proof}

\section{Spectral gaps inequalities}
\label{s3}
 With the comparison principle at our disposal, it is not very difficult to prove Theorems \ref{t1} and \ref{t3}. Recall that Theorem \ref{t2} is a consequence of Theorem \ref{t2} and Aldous' conjecture. Both proofs follow the same strategy: first we prove the lower bound for the transition rate $q(z) = |z|^{-(1+\alpha)}$ and then we use the comparison principle and the variational formula \eqref{varia} to extend this lower bound to general $p(\cdot) \in \mathrm{DAN}(\alpha)$. The upper bounds is obtained by choosing a suitable test function.

\begin{proof}[Proof of Theorem \ref{t2}]
Recall that the unique equilibrium of the process $y_n(\cdot)$ is the uniform measure on $\Lambda_n$. For $f: \Lambda_n \to \bb R$ such that $\int f d \mu_n =0$,
\[
\|f\|_{L^2(\mu)}^2 = \frac{1}{2n+1} \sum_{x \in \Lambda_n} f(x)^2
\]
and 
\[
\mc D(f) = \frac{1}{2n+1} \sum_{x, y \in \Lambda_n} p(y-x) \big(f(y)-f(x)\big)^2.
\]
We start proving Theorem \ref{t1} for $q(z)  = |z|^{-(1+\alpha)}$. Fix $n \in \bb n$ and let $f: \Lambda_n \to \bb R$ be such that $\sum_{x\in \Lambda_n} f(x) =0$. Then,
\[
\begin{split}
\frac{1}{2n+1} \sum_{x \in \Lambda_n} f(x)^2 
		&= \frac{1}{2(2n+1)^2} \sum_{x,y \in \Lambda_n} \big( f(x) -f(y) \big)^2\\
		&\leq  \frac{1}{2(2n+1)^2} \sum_{x,y \in \Lambda_n} \big( f(x) -f(y) \big)^2\\
		&\leq \frac{1}{2(2n+1)^2} \sum_{x,y \in \Lambda_n} (2n+1)^{1+\alpha} q(y-x) \big(f(y)- f(x)\big)^2,
\end{split}
\]
since $q(y-x) \geq (2n+1)^{-(1+\alpha)}$ for any $x,y \in \Lambda_n$. This proves the lower bound in Theorem \ref{t1} for $q(\cdot)$ with $\kappa_2 = 2$. Now define 
\[
\phi(x) = \sum_{x=-n}^{n-k} \big( f(x+k)-f(x)\big)^2 
\]
for $x=1,\dots, 2n$, and $\phi(x) =0$ otherwise. If $x+y \leq 2n$, 
\[
\begin{split}
\phi(x+y) 
		&=\sum_{z=-n}^{n-x-y} \big(f(z+x+y)-f(z)\big)^2 \\
		&\leq \sum_{z=-n}^{n-x-y} 2\Big\{\big(f(z+x+y)-f(z+x)\big)^2+\big(f(z+x)-f(z)\big)^2\Big\}\\
		&\leq 2\big(\phi(x) + \phi(y)\big).
\end{split}
\]
If $x+y >2n$, $\phi(x) = 0$ and we conclude that $\phi$ is $2$-subpolynomial. Therefore, we can apply Theorem \ref{t4} to prove that
\[
 \frac{1}{2n+1} \sum_{x, y \in \Lambda_n} q(y-x) \big(f(y)-f(x)\big)^2 \leq 
 \frac{\kappa}{2n+1} \sum_{x, y \in \Lambda_n} p(y-x) \big(f(y)-f(x)\big)^2,
\]
from where the lower bound of Theorem \ref{t1} follows for $p(\cdot)$ with $\kappa_2 = 2 \kappa$. The upper bound can be proved using $f(x) = \sqrt{\frac{2n+1}{n}} \mathbf{1}(x>0)$ as a test function.
\end{proof}

\begin{proof}[Proof of Theorem \ref{t3}]
For $x,y \in \Lambda_n$ and $f: \Omega_{n,\ell}^{\mathrm{zr}} \to \bb R$, define
\[
\phi_{n,\ell}^{x,y} (f) := \int g(\xi_x) \big( f(\xi^{x,y}) -f(\xi)\big)^2 d \mu_{n,\ell}^{\mathrm{zr}}.
\]
Notice that
\[
\mc D_{n,\ell}^{\mathrm{zr}}(f) = - \int f L_{n,\ell}^{\mathrm{zr}} f d \mu_{n,\ell}^{\mathrm{zr}} = \frac{1}{2} \sum_{x,y \in \Lambda_n} p(y-x) \phi_{x,\ell}^{x,y} (f).
\]
It is well-known (see Chapter 5.5 in \cite{KipLan} for example) that for any $x,y, z \in \Lambda_n$, 
\[
\phi_{n,\ell}^{x,z}(f)  \leq  2 \big( \phi_{n,\ell}^{x,y}(f) - \phi_{n,\ell}^{y,z}(f)\big).
\]
As in the proof of Theorem \ref{t1}, define for $x=1,\dots,2n$,
\[
\phi_{n,\ell}(x) = \sum_{y=-n}^{n-x} \phi_{n,\ell}^{x,x+k}(f)
\]
and $\phi_{n,\ell}(x)= 0$ for $x >2n$. Since $\phi$ is $2$-subpolynomial, the same computations performed above give the estimate
\[
\mc D_{n,\ell}^{\mathrm{zr}} (f) \geq \frac{1}{2\kappa} \sum_{z=1}^{2n} q(z) \phi_{n,\ell}(z)
		\geq \frac{1}{\kappa (2n+1)^{1+\alpha}} \cdot \frac{1}{2} \sum_{x,y \in \Lambda_n} \phi_{n,\ell}^{x,y}(f).
\]
This last sum is the Dirichlet form of the zero-range process on the complete graph. On the complete graph, under the conditions of part i) of Theorem \ref{t3}, it was proved in \cite{Cap} (see also \cite{LanSetVar}) that the spectral gap is of order $\mc O\big(n\big)$. Under the conditions of 
part i) of Theorem \ref{t3}, it was proved in \cite{Mor} that the spectral gap is of order $N(1+\rho)^{-2}$, so Theorem \ref{t4} is proved.
\end{proof}

\section{Applications}
\label{s4}

In this section we prove Corollary \ref{c1} and we also point out some other possible applications of the spectral gap inequalities proved here.

\subsection{Proof of Corollary \ref{c1}}

By reversibility of the random walk $y(\cdot)$, we have that
\[
\bb P^0\big(y(2t) = 0\big) = \sum_{x \in \bb Z} \bb P^0\big( y(t) =x\big)^2,
\]
so it is enough to prove that
\[
\psi(t) := \sum_{x \in \bb Z} \bb P^0\big( y(t) =x\big)^2 \leq \frac{c_1}{t^{1/\alpha}}
\]
for some finite constant $c_1$. Let us define $f_t(x) = \bb P^0\big(y(t) =x\big)$. We have that
\[
\tfrac{d}{dt} \psi(t)= -2 \sum_{x,y \in \bb Z} p(y-x) \big( f_t(y) -f_t(x) \big)^2.
\]
Notice that the right-hand side of this identity is very similar to $\mc D(f_t)$, except that the sum goes over all of $\bb Z$ instead of $\Lambda_n$ and we miss the prefactor $(2n+1)^{-1}$. It is also important to point out that $f_t$ is {\em not} a mean-zero function. We need to fix both problems in order to use Theorem \ref{t2}. Fix $n \in \bb N$ and define
\[
\nu_{n,x}(t) := \sum_{i=-n}^{n} f_t\big((2n+1) x +i\big).
\]
 Since $f_t$ is a transition probability, $\sum_{x \in \bb Z} \nu_{n,x}(t)=1$. Applying Theorem \ref{t1} to the function $f_t - \nu_{n,x}(t)$ restricted to the interval $\{x-n,\dots,x+n\}$ , we see that
 \begin{equation}
 \label{curico}
 \sum_{i=-n}^n f_t\big((2n+1)x+i\big)^2 \leq \kappa_1^{-1} (2n+1)^\alpha \mc D_{n,x}(f_t) + (2n+1)^{-2} \nu_{n,x}(t)^2,
 \end{equation}
 where
 \[
 \mc D_{n,x}(f_t) = \sum_{y,z=-n}^n p(z-y)\big(f_t(x+z)-f_t(x+y)\big)^2.
 \]
 Taking the sum with respect to $x$ in \eqref{curico}, we see that
 \[
\psi(t) \leq \kappa_1^{-1} (2n+1)^\alpha \sum_{x,y \in \bb Z} p(y-x) \big( f_t(y) -f_t(x) \big)^2
		+(2n+1)^{-1} \sum_{x \in \bb Z} \nu_{n,x}(t)^2.
 \]
Since $\nu_{n,x}(t) \leq 1$, the last sum above is bounded by $(2n+1)^{-1}$. Therefore,
 \[
 \tfrac{d}{dt} \psi(t) \leq -2\kappa_1 (2n+1)^{-\alpha} \psi(t) + 2\kappa_1 (2n+1)^{-(1+\alpha)}.
 \]
 Minimizing over $n$ we conclude that there exists a finite constant $c_2$ such that
 \[
 \tfrac{d}{dt} \psi(t) \leq - c_2 \psi(t)^{-(1+\alpha)}.
 \]
 Integrating in time this estimate, Corollary \ref{c1} is proved.

\subsection{Scaling limits of interacting particle systems}

In this section we mention some possible applications of the spectral gap estimates proved here. In \cite{GonJar}, fluctuation results were derived for the exclusion process with long jumps, including derivations of fractional heat and fractional KPZ equations. In Section 6.3 of \cite{GonJar}, it is discussed the role of the spectral gap inequality in the proof of such convergence results, and in particular Theorem \ref{t3} allows to extend the results of \cite{GonJar} to the case on which $p(\cdot) \in \mathrm{DAN}(\alpha)$. Similar results are obtained in \cite{Set} for the zero-range process with long jumps, where now the estimate of Theorem \ref{t4} is used as input, see condition (SG) in Section 2 of \cite{Set}. 

\section*{Acknowledgements}

\noindent
M.J.~would like to thank the warm hospitality of the Isaac Newton Institute of Mathematical Sciences, where this work was finished.
M.J.~acknowledges CNPq for its support through the Grant 305075/2017-9 and ERC for its support through the European Unions Horizon 2020 research and innovative programme (Grant Agreement No. 715734).

\bibliographystyle{plain}

\end{document}